\documentclass[10pt]{amsart}
\usepackage{fullpage}   

\usepackage{amsmath}
\usepackage{amsfonts}
\usepackage{amssymb}
\usepackage{amsthm}
\usepackage{color}
\usepackage{hyperref}

\newtheorem{theorem}{Theorem}
\newtheorem{proposition}{Proposition}
\newtheorem{lemma}{Lemma}
\newtheorem{corollary}{Corollary}

\theoremstyle{definition}
\newtheorem{definition}{Definition}
\newtheorem{conjecture}{Conjecture}

\theoremstyle{remark}
\newtheorem{remark}{Remark}

\newcommand{\s}{\mathcal{S}}
\newcommand{\e}{\mathrm{e}^}

\newcommand{\BN}{\mathbb{N}}

\newcommand{\Q}{\mathbb{Q}}
\newcommand{\Z}{\mathbb{Z}}

\DeclareMathOperator{\col}{Col}
\DeclareMathOperator{\com}{\mathcal{C}}

\begin{document}



\title[]{The Potts model and chromatic functions of graphs}

\author{Martin Klazar$^*$}
\author{Martin Loebl$^*$}
\author{Iain Moffatt$^\dagger$}
\address{$^*$Department of Applied Mathematics, Charles University, Malostransk\'{e} nam. 25, 118 00 Praha 1, Czech Republic.}
\address{$^\dagger$ Department of Mathematics, Royal Holloway, University of London, Egham, Surrey, TW20 0EX, United Kingdom}
\email{klazar@kam.mff.cuni.cz; loebl@kam.mff.cuni.cz; iain.moffatt@rhul.ac.uk}

\thanks{Martin Klazar and Martin Loebl were partially supported by the
Czech Science Foundation GACR under the contract number P202-12-G061, CE-ITI}

\date{\today}

\begin{abstract}  
The $U$-polynomial of Noble and Welsh  is known to have intimate connections with the Potts model as well as with several important graph polynomials.  For each graph $G$,  $U(G)$ is equivalent to  Stanley's symmetric bad colouring polynomial $XB(G)$. Moreover Sarmiento established the equivalence between $U$ and the polychromate of Brylawski.  Loebl defined the $q$-dichromate $B_q(G,x,y)$ as a function of a graph $G$ and three independent variables $q,x,y$, proved that it is equal to the partition function of the Potts model with variable number of states and with a certain external field contribution, and conjectured that the $q$-dichromate is equivalent to the $U$-polynomial. He also proposed a stronger conjecture on integer partitions. The aim of this paper is two-fold. We present a construction disproving  Loebl's integer partitions conjecture, and we introduce a new function  $B_{r,q}(G;x,k)$ which is also equal to the partition function of the Potts model with variable number of states and with a (different) external field contribution, and we show that
$B_{r,q}(G;x,k)$ is equivalent to the $U$-polynomial and to Stanley's symmetric bad colouring polynomial. 
\end{abstract}

\maketitle


\section{Introduction}
\label{sec.intro}

Here we are interested in two families of   generalisations of the chromatic polynomial of graphs. The first family contains  polynomials  associated with symmetric functions, and the second family arises from quantum knot theory. Both of these families are known to have intimate connections with the Potts model. 
This paper provides a step towards a full understanding of the connections between these two families of graph polynomials and their connections with the Potts model, and towards the development of a framework for the application of the statistical mechanical `toolbox' to address hard problems in the theory of these generalised chromatic polynomials. Such statistical mechanical approaches have proven to be very effective in graph theory, particularly for   the study of the zeros of the chromatic polynomial (see \cite{Sok99}, for example). 

Graphs here can have  multiple edges or loops. If $G=(V,E)$ is a graph, $G-e$ (respectively,  $G/e$)  denotes the graph obtained from $G$ by deleting (respectively,  contracting)   an edge $e\in E$. If $A \subseteq E$, then  $\com(A)$ denotes the set of connected components of the spanning subgraph $(V,A)$ of $G$,  $k(A)=|\com(A)|$, and  $|C|$ denotes the number of vertices in $C\in \com(A)$.
A {\em proper $k$-colouring} of a graph $G=(V,E)$ is a mapping $s:\;V\rightarrow \{0,\ldots, k-1\}$ with the property 
that $s(u)\neq s(v)$ for all $uv\in E$. We use $\col(G;k)$ to denote the set of proper $k$-colourings of $G$, and use $\col(G)=
\bigcup_{k=0}^{\infty}\col(G;k)$ to denote the set of all proper colourings of $G$.

\medskip

 We are  interested in knowing when one graph function 
$Q= (Q(G); G$ graph$)$  determines another graph function $P= (P(G); G$ graph$)$. 
We say that graph function $P$ {\em determines} graph function $Q$ if, for every $G$, any evaluation of $Q(G)$ can be obtained given an oracle which returns evaluations of $P(G)$.  We emphasise that the graph $G$ is not known when computing evaluations of $Q(G)$. Clearly, if $P(G)$ determines $G$ uniquely then it also determines $Q(G)$. 
We also note that in the case that $Q(G)$ is a polynomial for each $G$, then $P$ determines $Q$ if for each $G$ the evaluations of $P(G)$ determine the coefficients of $Q(G)$ expressed in some specified basis.
 We say that $P$ and $Q$  are {\em equivalent} if they determine each other.

\medskip

We  now give a brief overview of our first family of graph polynomials. 
Stanley's {\em symmetric function generalisation of the chromatic polynomial} (see \cite{S1}, \cite{S2}) is defined by 
$$
X(G;x_0,x_1,\dots):=\sum_{s\in\col(G)}\prod_{v\in V}x_{s(v)}.
$$
We are particularly interested in a generalisation of this, the {\em symmetric function generalisation of the bad colouring polynomial} of \cite{S1}, which is defined by
$$
XB(G;t,x_0,x_1,\dots):= \sum_{s:\;V\to\{0,1,\dots\}}(1+t)^{b(s)}\prod_{v\in V}x_{s(v)},
$$
where the sum ranges over all (not necessarily proper) colourings of $G$ by the colours $\{0,1,\dots\}$, and 
$b(s):=|\{uv\in E: s(u)= s(v)\}|$ denotes the number of monochromatic edges of $s$. Note that $X(G)$ can be recovered from this polynomial.

In \cite{NW}, Noble and Welsh introduced  the  $U$-polynomial and showed that it is equivalent to $XB(G)$.
\begin{definition}
\label{def.u}
Let $G$ be a graph, the {\em U-polynomial} is  
$$
U(G;z,x_1,x_2,\dots):=\sum_{A\subseteq E(G)}x(\tau_A)(z-1)^{|A|-|V|+ k(A)},
$$
where $\tau_A= (n_1\geq n_2 \geq \dots \geq n_{l(A)})$ is the partition of the number $|V|$ determined by the  sizes of the $l(A)$ connected components of the spanning subgraph $(V,A)$ of $G$ (so  the $i$-th component has $n_i$ vertices), 
and $x(\tau_A)= x_{n_1}\dots x_{n_l}$.
\end{definition} 
Sarmiento proved in \cite{Sa} that the  $U$-polynomial, and hence $XB$, is also equivalent to  Brylawski's  polychromate from  \cite{Br}. In \cite{EMM12}, it was shown that $U(G)$ has a close connection with the Potts model (see Section~\ref{pm} for a discussion of the Potts model):
for $x_i= ( \sum_{j=1}^k  \e{i\beta  H_j} )/(\e{\beta J}-1)$,
\begin{equation}\label{upotts}
U(G,k, x_0, x_1, \dots)= (\e{\beta J}-1)^{|V|} \sum_{s:V\rightarrow \{0,\ldots, k-1\}} \e{ \sum_{v\in V}\beta H_{s(v)} }  e^{\beta E(s)},
\end{equation}
where $E(s)= \beta J \sum_{uv\in E(G)}\delta(s(u), s(v))$ is the Potts model energy of state $s$ when all the coupling constants are equal to $J$; $H_{s(v)} $ is the energy contribution of an external field for a site in state $s(v)$; $\delta$ is the Kronecker delta function defined by $\delta(a,b)= 1$ if $a=b$, and $\delta(a,b)=0$ otherwise.

\medskip

We now turn our attention to our second family of graph polynomials. This family of polynomials arises from the theory of Vassiliev and quantum knot invariants (see \cite{Lo07}).  In what follows, if $q$ is a variable then its range is the non-negative integers. The ranges of all of the other variables appearing in this paper are the reals.  Generally,  $k$ denotes a non-negative integer. 

The {\em $q$-chromatic function}, \cite{Lo07}, of a graph $G=(V,E)$ is 
\begin{equation*}
M_q(G;k):=\sum_{s\in \col(G;k)} q^{\sum_{v\in V} s(v)}.
\end{equation*}
It  was shown in \cite{Lo07} that the $q$-chromatic function can be  expressed as a sum over spanning subgraphs:
$$
M_q(G,k)=\sum_{A\subseteq E} (-1)^{|A|} \prod_{W\in\com(A)} (k)_{q^{|W|}},
$$
where $\com(A)$ denotes the set of connected components of the spanning subgraph $(V,A)$, $|W|$ denotes the number of vertices in the component $W$, and $(k)_n=k(k-1)\dots(k-n+1)$. This expression leads naturally to the  {\em q-dichromate} which is defined as
$$
B_q(G,x,y)=\sum_{A\subseteq E} x^{|A|}\prod_{W\in\com(A)} (y)_{q^{|W|}}.
$$
 It was shown in \cite{Lo07} that for each real $J$,
\begin{equation}\label{bpotts}
B_q(G,e^{J}-1,k)=\sum_{s:V\rightarrow \{0,\ldots, k-1\}}q^{\sum_{v\in V}s(v)}e^{E(P^k)(s)},
\end{equation}
where $E(P^k)(s)= \sum_{uv\in E(G)}J\delta(s(u), s(v))$ is the Potts model energy of state $s$ when all the coupling constant are equal to $J$.

\medskip

Comparing  Equations~\ref{upotts} and \ref{bpotts} gives that both families of polynomials are closely connected to the Potts model partition function. Moreover, it becomes apparent that the two families are closely related to each other. In fact, it is easily seen that  for a positive integer $k$,
$$
x^{-|V|}B_q(G,x,k)=  U(G;z,x_1,x_2,\dots)|_{z:= x+1, x_i:= x^{-1}(k)_{q^i}},
$$ 
so every polynomial in the second family can be obtained from $U$, $XB$, and from Brylawski's polychromate. In  \cite{Lo07}, the second author conjectured that the converse holds:
\begin{conjecture}[\cite{Lo07}]\label{c.qu}
The $q$-dichromate is equivalent to the $U$-polynomial.
\end{conjecture}
 The significance of Conjecture~\ref{c.qu} is that, if it is true, it would allow the polynomials in the first family, each of which requires an infinite number of variables, to be written as a natural function in finitely many variables. 
In fact, a conjecture stronger that Conjecture~\ref{c.qu} was made in \cite{Lo07}.  For this conjecture,  let $\tau= (n_1\geq n_2\geq\dots\geq n_k)$ be a partition 
of $n$. We let $c(\tau)= (c(\tau, y))_{y=0,1,\dots}$ be an infinite sequence of polynomials in $q$ defined by 
$$
c(\tau, y)= \prod_{i=1}^k(q^{n_iy}+q^{n_i(y-1)}+\dots +1).
$$
\begin{conjecture}[\cite{Lo07}]
\label{c.ekv}
Only a trivial rational linear combination of sequences  $c(\tau)$ is identically zero.
\end{conjecture}
It is not difficult to observe that Conjecture \ref{c.ekv} implies that the $q$-dichromate determines the $U$-polynomial, and hence resolves Conjecture \ref{c.qu}.

In this paper we resolve  Conjecture \ref{c.ekv} in the negative: we show that it is false if $n\ge 39$.   We then go on to define and study a variant of the $q$-dichromate, namely $B_{r,q}(G;x,k)$, describe its properties and finally establish that $B_{r,q}(G;x,k)$ is equivalent to the $U$-polynomial. This shows that  Conjecture \ref{c.qu} holds for a variant of $B_q$.

\section{Dependency of integer partitions}
\label{sec.dep}

In this section we examine Conjecture~\ref{c.ekv}.
We use $\tau\vdash n$ to denote that $\tau=(n_1\geq n_2\geq\dots\geq n_k)$ is a partition of $n$, that is, 
the $n_i$ are positive
integers summing up to $n$, listed in non-increasing order. We disprove 
Conjecture \ref{c.ekv} by showing that if  $n$ is large enough (as we will see, $n\ge39$
suffices) then there exist fractions $\{\alpha_{\tau}\;|\;\tau\vdash n\}$, not all of them zero, such that for every
$y=0,1,2,\dots$,
$$
\sum_{\tau\vdash n}\alpha_{\tau}\cdot c(\tau,y)
$$
is a zero polynomial in $q$. It is helpful to collect the polynomials $c(\tau,y)$ in the matrix $M(n)$, where
$
M(n)=(c(\tau,y))_{\tau,y},
$
with $p(n)$ rows indexed by all partitions $\tau$ of $n$ and infinitely many columns indexed by the non-negative integers
$y$. Here $p(n)$ denotes the number of partitions of $n$. Let $\mathrm{rank}_{\Q}M(n)$ be the rank of $M(n)$
over the field $\Q$, that is, the maximum number of linearly independent rows. Our disproval of Conjecture \ref{c.ekv} 
has the following quantitative form. 

\begin{proposition}\label{refut}
For every $n=1,2,\dots$,
$$
\mathrm{rank}_{\Q}M(n)\le\frac{n^3+n+2}{2}\;.
$$
In other words, if $m>\frac{n^3+n+2}{2}$ and $\tau(1),\tau(2),\dots,\tau(m)$ are $m$ distinct partitions of $n$, then there exist fractions $\alpha_1,\alpha_2,\dots,\alpha_m$, not all of them zero, such that for every $y=0,1,2,\dots$,
$$
\sum_{i=1}^m\alpha_i\cdot c(\tau(i),y)
$$
is a zero polynomial in $q$. In particular, if $p(n)>\frac{n^3+n+2}{2}$, which is true for every $n\ge n_0$, 
Conjecture \ref{c.ekv} does not hold. 
\end{proposition}
\begin{proof}
We have
$$
c(\tau,y)=\frac{\prod_{i=1}^k (q^{n_i(y+1)}-1)}{\prod_{i=1}^k (q^{n_i}-1)}\;.
$$
We regard $y$ as a formal variable, along with $q$, and introduce another formal variable $z=q^{y+1}$. 
Then
$$
c(\tau,q,z)=\frac{\prod_{i=1}^k(z^{n_i}-1)}{\prod_{i=1}^k(q^{n_i}-1)}
$$
is a rational function in $\Q(q,z)$. The polynomial
$
P(q)=(q-1)(q^2-1)\dots(q^n-1)
$
is divisible, in $\Z[q]$, by any of the denominators $(q^{n_1}-1)(q^{n_2}-1)...(q^{n_k}-1)$. This follows from the fact that the multiplicity of any root of the denominator polynomial, 
which is a primitive $r$-th root of unity for some $r\le n$, is majorized by its multiplicity
as a root of $P(q)$: the former multiplicity is equal to  the number of parts $n_i$ in $\tau$ that are divisible 
by $r$, which is at most $\lfloor n/r\rfloor$, and $\lfloor n/r\rfloor$ is exactly equal to the latter multiplicity as 
$\lfloor n/r\rfloor$ counts the multiples of $r$ among $1,2,\dots,n$. Denoting the  polynomial  $ P(q)/ \prod_{i=1}^k(q^{n_i}-1) $ by $a(\tau,q)$,   the numerator polynomial $(z^{n_1}-1)(z^{n_2}-1)\dots(z^{n_k}-1)$ by $b(\tau,z)$, and setting
$$
s(\tau,q,z)=P(q)\cdot c(\tau,q,z)=a(\tau,q)\cdot b(\tau,z)\;, 
$$
we see that
$
S(n):=\{s(\tau,q,z)\;|\;\tau\vdash n\}
$
is a set of $p(n)$ polynomials $s(\tau)\in\Q[q,z]$, each of which factors into the product of $a(\tau)\in\Z[q]$ with 
degree $1+2+\dots+n-n=\binom{n}{2}$ and $b(\tau)\in\Z[z]$ with degree $n$. The linear 
dimension of $S(n)$ as a subset of the vector space $\Q[q,z]$ over $\Q$ is therefore bounded by
$$
\dim_{\Q} S(n)\le(1+\binom{n}{2})(1+n)=\frac{n^3+n+2}{2}\;,
$$
because each $s(\tau)\in S(n)$ is a $\Z$-linear combination of the monomials $q^iz^j$ with $0\le i\le\binom{n}{2}$ 
and $0\le j\le n$. Thus more than $\frac{n^3+n+2}{2}$ polynomials $s(\tau,q,z)$ are always linearly dependent. 
The $\tau$-th row of $M(n)$ equals
$$
P(q)^{-1}\cdot(s(\tau,q,q),s(\tau,q,q^2),s(\tau,q,q^3),\dots)\;.
$$
Hence more than $\frac{n^3+n+2}{2}$ rows are linearly dependent too (because any linear combination of the polynomials
$s(\tau,q,z)$ is preserved by the substitutions $z=q^{y+1}$, $y=0,1,2,\dots$), and this quantity bounds the rank of the matrix $M(n)$ over $\Q$.
\end{proof}

\noindent
Thus Conjecture~\ref{c.ekv} is false whenever the partition function $p(n)$ satisfies
$
p(n)>\frac{n^3+n+2}{2}
$. 
Since $p(n)$ is at least $\frac{1}{5!}$ of the number of quintuples $(a_1,a_2,\dots,a_5)\in\mathbb{Z}_+^5$  with $a_1+a_2+\dots+a_5=n$, 
that is, $p(n)\ge\frac{1}{120}\binom{n-1}{4}\gg n^4$, the  inequality $p(n)>\frac{n^3+n+2}{2}$    holds for some sufficiently large $n$. 
In fact, the computer algebra package Maple can be used to show that the inequality holds if and only if $n\ge 39$.

A disproval of Conjecture~\ref{c.ekv} was published first in the technical report \cite{klaz_loeb} in a more complicated
form. In \cite{klaz_loeb}, M. Loebl offered a weaker conjecture that the quantities $c(\tau,y,q)$, $\tau\vdash n$, are 
linearly independent over $\Q$ when regarded as bivariate functions of the real variables $q,y\ge 0$. 
However, the proof of Proposition~\ref{refut} works without change in this situation too and shows that  the 
weaker conjecture also fails to hold. A natural question is what is the smallest integer $n_0$ for which Conjecture~\ref{c.ekv}
is not valid. We have proven above that $n_0\le39$, and it is  shown in \cite{klaz_loeb} that $n_0\ge7$. We hope to address
this and other questions related to Conjecture~\ref{c.ekv} elsewhere.

\section{Deformations of the chromatic polynomial}
\label{sec.deff}

In this section we define a variant of the $q$-dichromate, and show that it is equivalent to the $U$-polynomial. 

\subsection{The $(r,q)$-chromatic function}

\begin{definition}\label{d.1}
Let $k\in\mathbb{N}$.  We define the {\em $(r,q)$-chromatic function} as
\begin{equation*}
M_{r,q}(G;k):=\sum_{s\in \col(G;k)}  r^{\sum_{v\in V} q^{s(v)}}.
\end{equation*}
\end{definition}
The chromatic polynomial can be recovered as $M_{1,q}(G;k)$, and the $q$-chromatic function $M_r(G,k)$ can be obtained from $M_{r,q}(G;k)$ by replacing each $q^i$ with its exponent $i$. 

Just as for $B_q$,  the function $M_{r,q}$ can be written as a sum over spanning subgraphs:
\begin{proposition}\label{p.1} Let $G=(V,E)$ be a graph. Then 
\begin{equation}\label{e.p1}    
M_{r,q}(G;k)=\sum_{A\subseteq E(G)}  (-1)^{|A|}  \prod_{C\in \com(A)}   \sum_{i=0}^{k-1}   r^{|C| q^{i}},
\end{equation}

\end{proposition}
Although Proposition~\ref{p.1} can be proven directly using the inclusion-exclusion principle, here we will give an alternative poof that uses vertex-weighted graphs. The advantage of this  approach, as we will see, is that it allows for more general results.  We therefore postpone the proof of  Proposition~\ref{p.1} until after we introduce and discuss a vertex-weighted analogue of $M_{r,q}(G;k)$.

Here, a {\em vertex-weighted graph} consists of a graph $G=(V,E)$ and a  {\em weight function} $\omega:  V\rightarrow \mathbb{N}$.  The {\em weight} of the vertex $v$ is the value $\omega (v)$. 
If $G$ is a vertex-weighted graph with weight function $\omega$, and $e$ is an edge of $G$, then $G-e$ is the vertex-weighted graph obtained by deleting the edge $e$ of $G$ and leaving the weight function unchanged. If $e$ is any non-loop edge of $G$, then $G/e$ is the vertex-weighted graph obtained from $G$ by contracting the edge $e$ and changing  the vertex weight function as follows: if $u$ and $v$ are the vertices incident to $e$, and $w$ is the vertex of $G/e$ created by the contraction, then $ \omega(w) :=  \omega(u) +\omega(v) $. All other vertices of $G/e$ have the same weight as the corresponding vertex in $G$.   Loops are not contracted.

\begin{definition}\label{d.2}
Let $k\in\mathbb{N}$, and let $G=(V,E)$ be a vertex-weighted graph with weight function $\omega$.  We define the {\em $(r,q)$-chromatic function} as
\begin{equation*}
M_{r,q}^{\omega}(G;k):=\sum_{s\in \col(G;k)}  r^{\sum_{v\in V}  \omega(v) q^{s(v)}}.
\end{equation*}
\end{definition}

Observe that if $G$ is a graph, and $G_1$ is the vertex-weighted graph obtained from $G$ be giving each vertex weight 1, then $M_{r,q}(G;k)=M_{r,q}^{\omega}(G_1,k)$.

The main advantage of considering the vertex-weighted polynomial, is that $M_{r,q}$ satisfies a deletion-contraction identity.
\begin{proposition}\label{p.2} 
Let $k\in\mathbb{N}$, and let $G=(V,E)$ be a vertex-weighted graph with weight function $\omega:V\rightarrow \mathbb{N}$. 
Then, if $e\in E$ is not a loop, 
\begin{equation*}
M_{r,q}^{\omega}(G;k)=M_{r,q}^{\omega}(G-e;k) -M_{r,q}^{\omega}(G/e;k). 
\end{equation*}
\end{proposition}
\begin{proof}
Let $e=ab$ be a non-loop edge of $G$, and let $V=V(G)=V(G-e)$. By collecting together the colourings $s\in \col(G;k)$  in which $s(a) \neq s(b)$, and those in which $s(a) = s(b)$, we can write
\begin{equation}\label{e.mp}
M_{r,q}^{\omega}(G-e;k)=\sum\limits_{\stackrel{s\in \col(G-e;k)}{  s(a) \neq s(b)   }} r^{\sum_{v\in V}  \omega(v) q^{s(v)}}
+\sum\limits_{\stackrel{s\in \col(G-e;k)}{ s(a) = s(b) }}  r^{\sum_{v\in V}  \omega(v) q^{s(v)}}.
\end{equation} 
Each $s\in \col(G-e;k)$ with  $s(a) \neq s(b)$ induces (by setting $s'(w)=s(w)$ for each vertex $w$) a proper colouring $s'\in \col(G;k)$, and as $G$ and $G-e$ have the same vertex set and weight function, we can write the first sum on the right-hand side of Equation~\ref{e.mp} as $M_{r,q}^{\omega}(G;k)$.  Similarly, if $c$ is the vertex in $G/e$ obtained by contracting $e=ab$, each $s\in \col(G-e;k)$ with  $s(a) = s(b)$ induces a proper colouring $s'\in \col(G/e;k)$ by setting $s'(c)=s(a)$ and $s'(w)=s(w)$ for all of the other vertices $w$. Then, using the facts that   $\omega(c)=\omega(a)+\omega(b)$ and $s(a)=s(b)$,
\begin{multline*}
\sum\limits_{\stackrel{s\in \col(G-e;k)}{ s(a) = s(b) }}  r^{\sum_{v\in V}  \omega(v) q^{s(v)}}  
= \sum\limits_{\stackrel{s\in \col(G-e;k)}{ s(a) = s(b) }}   r^{\omega(a)q^{s(a)} +\omega(b)  q^{s(b) }}   r^{\sum_{v\in V\backslash \{a,b\}  }  \omega(v) q^{s(v)}}  
\\
=\sum\limits_{s\in \col(G/e;k)}    r^{\omega(c)q^{s(c)}} r^{\sum_{v\in V(G/e)\backslash \{c\}     }  \omega(v) q^{s(v)}} 
= M_{r,q}^{\omega}(G/e;k).
\end{multline*}
Thus $M_{r,q}^{\omega}(G-e;k)=M^{\omega}_{r,q}(G;k)  + M_{r,q}^{\omega}(G/e;k)$, as required.
\end{proof}

It is easy to verify that $M_{r,q}^{\omega}$ is multiplicative under disjoint unions of graphs. This fact, together with Proposition~\ref{p.2}; the evaluation $M_{r,q}^{\omega}(N;k) = \sum_{i=0}^{k-1} r^{\omega(v) q^i }$, for $N=(\{v\},\emptyset)$; and the fact that  $M_{r,q}^{\omega}(G;k)=0$ if $G$ contains a loop, provides a recursive definition for $M^{\omega}_{r,q}$.

The following result generalises Proposition~\ref{p.1}
\begin{proposition}\label{p.3} Let $k\in\mathbb{N}$, and let $G=(V,E)$ be a vertex-weighted graph with weight function $\omega:V\rightarrow \mathbb{N}$.  Then 
\begin{equation*}
M_{r,q}^{\omega}(G;k)=\sum_{A\subseteq E}  (-1)^{|A|}  \prod_{C\in \com(A)}   \sum_{i=0}^{k-1}   r^{\omega(C) q^{i}},
\end{equation*}
where $\com(A)$ denotes the set of connected components of the spanning subgraph $(V,A)$, and $\omega(C)$ denotes the total weight of the component $C$, that is, $\omega(C) = \sum_{v\in V(C)} \omega(v) $.
\end{proposition}
\begin{proof}
Let
$ S(G):= \sum_{A\subseteq E}  (-1)^{|A|}  \prod_{C\in \com(A)}   \sum_{i=0}^{k-1}   r^{\omega(C) q^{i}}$.
We prove that $S(G)=M_{r,q}^{\omega}(G;k)$ by induction on the number of non-loop edges of $G$. 
The claim is readily verified if $G$ has no non-loop edges. Suppose that the claim is true for all graphs with $m>1$ non-loop edges. Let $G$ be a graph with $m+1$ non-loop edges, and  $e$ be a non-loop edge of $G$. Then  
\[
S(G)= \sum_{\stackrel{A\subseteq E(G)}{e\notin A} }  (-1)^{|A|}  \prod_{C\in \com(A)}   \sum_{i=0}^{k-1}   r^{\omega(C) q^{i}}
+ \sum_{\stackrel{A\subseteq E(G)}{e\in A} }  (-1)^{|A|}  \prod_{C\in \com(A)}   \sum_{i=0}^{k-1}   r^{\omega(C) q^{i}}.
\]
There is a natural bijection between the spanning subgraphs of $G$ that contain the edge 
$e$ and the spanning subgraphs of $G/e$ (given by $A\mapsto A\backslash \{e\}$). Using this correspondence and the obvious correspondence between the spanning subgraphs of $G-e$ and the spanning subgraphs of $G$ that do not contain $e$, we can write the above expression as
\[
S(G)= \sum_{A\subseteq E(G-e)}  (-1)^{|A|}  \prod_{C\in \com(A)}   \sum_{i=0}^{k-1}   r^{\omega(C) q^{i}}
+(-1)  \sum_{A\subseteq E(G/e)}  (-1)^{|A|}  \prod_{C\in \com(A)}   \sum_{i=0}^{k-1}   r^{\omega(C) q^{i}},
\]
and it follows by the inductive hypothesis and Proposition~\ref{p.2} that $S(G)=   M_{r,q}(G-e;k) -M_{r,q}(G/e;k)=M_{r,q}(G;k) $, as required.
\end{proof}

\begin{proof}[Proof of Proposition~\ref{p.1}]
The result follows as the special case of Proposition~\ref{p.3} when $\omega(v)=1$ for all $v\in V$. 
\end{proof}

The expansion in Equation~\ref{e.p1} motivates the introduction of the following polynomial.
\begin{definition}\label{d.3}
Let $k\in\mathbb{N}$, and let $G=(V,E)$ be a vertex-weighted graph with weight function $\omega$.  We define the {\em $(r,q)$-dichromatic function} as
\begin{equation*}
B_{r,q}(G;x,k):=\sum_{A\subseteq E(G)}  x^{|A|}  \prod_{C\in \com(A)}   \sum_{i=0}^{k-1}   r^{|C| q^{i}}.
\end{equation*}
\end{definition}
Just as with $M_{r,q}$, the polynomial  $B_{r,q}$ is best understood through its extension to vertex-weighted graphs. Accordingly, if $G=(V,E)$ is a vertex-weighted graph with weight function $\omega:V\rightarrow \mathbb{N}$ then we set
\begin{equation*}
B_{r,q}^{\omega}(G;x,k):=\sum_{A\subseteq E(G)}  x^{|A|}  \prod_{C\in \com(A)}   \sum_{i=0}^{k-1}   r^{\omega(C) q^{i}}.
\end{equation*}
We have that  $B^{\omega}_{r,q}$  and $B_{r,q}$  agree when $G$ is a vertex-weighted graph in which each vertex has weight 1.

$B^{\omega}_{r,q}$ satisfies a deletion-contraction definition:
\begin{lemma}\label{p.4} 
Let $k\in\mathbb{N}$, and let $G=(V,E)$ be a vertex-weighted graph with weight function $\omega:V\rightarrow \mathbb{N}$. 
Then if $e\in E$ is not a loop, 
\begin{equation}\label{e.p4a}    
B_{r,q}^{\omega}(G;x,k)=B_{r,q}^{\omega}(G-e;x,k) +x B_{r,q}^{\omega}(G/e;x,k);
\end{equation}
and if $e\in E$ is  a loop, 
\begin{equation}\label{e.p4b}    
B_{r,q}^{\omega}(G;x,k)=(x+1) B^{\omega}_{r,q}(G-e;x,k).
\end{equation}
Furthermore, $B_{r,q}$ is multiplicative under disjoint unions of graphs, and $B_{r,q}^{\omega}(N;k) = \sum_{i=0}^{k-1} r^{\omega(v) q^i }$ when $N=(\{v\},\emptyset)$.
\end{lemma}
\begin{proof}
Let $e$ be a non-loop edge of $G$. Then  
\begin{multline*}
B_{r,q}^{\omega}(G;x,k)= \sum_{\stackrel{A\subseteq E(G)}{e\notin A} }  x^{|A|}  \prod_{C\in \com(A)}   \sum_{i=0}^{k-1}   r^{\omega(C) q^{i}}
+ \sum_{\stackrel{A\subseteq E(G)}{e\in A} }  x^{|A|}  \prod_{C\in \com(A)}   \sum_{i=0}^{k-1}   r^{\omega(C) q^{i}}
\\= \sum_{A\subseteq E(G-e)}  x^{|A|}  \prod_{C\in \com(A)}   \sum_{i=0}^{k-1}   r^{\omega(C) q^{i}}
+x \sum_{A\subseteq E(G/e)}  x^{|A|}  \prod_{C\in \com(A)}   \sum_{i=0}^{k-1}   r^{\omega(C) q^{i}}
= B_{r,q}^{\omega}(G-e;x,k)+ xB_{r,q}^{\omega}(G/e;x,k),
\end{multline*}
giving Equation~\ref{e.p4a}. Equation~\ref{e.p4b} can be proved in a similar way, and the remaining properties are easily verified.
\end{proof}

\subsection{The Potts model}\label{pm}

Let $G=(V,E)$ be a graph, and consider a set $\{0,1,\ldots, k-1\}$ of $k$ elements called \emph{spins}.  
A \emph{state} of a graph $G$ is an assignment of a single spin to each vertex
of the graph. (A state is exactly a vertex colouring, but  we use the term spin  as it is standard for the Potts model.) 
The $k$-state Potts model partition function at temperature $T$ in a site dependent  external field (see, for example, \cite{Wu82}) is  given by 
\begin{equation*}
Z(G) = \sum_{\sigma : V\rightarrow \{0, \ldots, k-1 \}}   \e{-\beta h(\sigma)},  \end{equation*}
with  Hamiltonian  
\begin{equation*}
h(\sigma) =  - J \sum_{ uv \in E(G) }   \delta(\sigma(u), \sigma(v))  - \sum_{v \in V(G)}    \sum_{i =0}^{k-1} H_{v,i} \, \delta(i,  \sigma(v)).
\end{equation*}
Here the external field at each vertex $v$ is specified by an ordered list $(H_{v,0},\ldots , H_{v,k-1})$, and  the external field contributes $H_{v,\sigma(v)}$ for a vertex $v$ with spin $\sigma(v)$; 
 $J$ is the spin-spin coupling; $\beta=1/(\kappa T)$, where $T$ is the temperature and $\kappa$ is the Boltzmann constant; and $\delta$ is the Kronecker delta function.

We will now relate $B_{r,q}(G;x,k)$ and the Potts model in an external field. 
\begin{lemma}\label{p.5} Let $G=(V,E)$ be a  vertex-weighted graph with weight function $\omega:V\rightarrow \mathbb{N}$, and let $k\in \mathbb{N}$ and  $x= \e{\beta J}-1 $. Then 
\begin{equation}\label{e.p5a2}
B^{\omega}_{r,q}(G;x,k) =  \sum_{\sigma : V\rightarrow \{0, \ldots, k-1 \}}       \e{\beta  \sum_{ uv \in E(G) } J   \delta(\sigma(u), \sigma(v)) }  r^{\sum_{v\in V} q^{\omega(v)  \sigma(v)}}
\end{equation} 
\end{lemma}
\begin{proof}
By adapting the proof of Proposition~\ref{p.2} it is easy to show that the right-hand side of Equation~\ref{e.p5a2} satisfies all of the identities given in Lemma~\ref{p.4} for  $x= \e{J\beta}-1 $. As these identities define $B_{r,q}(G;x,k)$, it follows that the two sides of Equation~\ref{e.p5a2} are equal.
\end{proof}

We immediately have the following.
\begin{corollary}
Let $G=(V,E)$ be a  graph, and let $k\in \mathbb{N}$, and $x= \e{\beta J}-1 $. Then 
\begin{equation}\label{e.p5a1}
B_{r,q}(G;x,k) =  \sum_{\sigma : V\rightarrow \{0, \ldots, k-1 \}}       \e{\beta  \sum_{ uv \in E(G) }   J \delta(\sigma(u), \sigma(v)) }  r^{\sum_{v\in V} q^{\sigma(v)}}.
\end{equation}
\end{corollary}

Observe that taking  $r=\e{}$ in Equations~\ref{e.p5a2} and \ref{e.p5a1} relates  $B^{\omega}_{r,q} $ and $B_{r,q}$  to the Potts model partition function in an external field. The external field contributes   $q^{\sigma(v)}$ for a vertex $v$ with spin $\sigma(v)$ in the unweighted case, and in the weighted case it contributes $\omega(v)q^{\sigma(v)}$, where $\omega(v)$ is the weight of $v$.

\begin{remark}
An alternative perspective on Lemma~\ref{p.5}  can be obtained by observing that it is almost the main result of \cite{EMM12} which expressed the Potts model partition function in an external field as an evaluation of the $\boldsymbol{V}$-polynomial (and thus extended the seminal relationship between the Tutte polynomial and the zero-field Potts model). In fact,  $B_{r,q}(G;x,k)$ can be obtained as an evaluation of the $\boldsymbol{V}$-polynomial and several of the results here can be deduced from this fact (and so can other properties such as a spanning tree expansion using \cite{MM12}). Here, however, we avoid this approach as we feel the extra notation it requires is a distraction  from the main purposes of this paper. 
\end{remark}

\subsection{Equivalence of $U$ and $B_{r,q}$}
\label{sub.u}
We show that $B_{r,q}$ is equivalent to Stanley's symmetric function generalisation of the bad colouring polynomial, $XB$.  
\begin{theorem}
\label{thm.eqq}
Let $r> 1$ be given. Then the graph functions $B_{r,q}$ and $XB$ are equivalent.
\end{theorem}
\begin{proof}
First, we have
\begin{multline*}
B_{r,q}(G;t-1,k) =  \sum_{\sigma : V\rightarrow \{0, \ldots, k-1 \}}       t^{\sum_{ uv \in E(G) }   \delta(\sigma(u), \sigma(v)) }  r^{\sum_{v\in V} q^{\sigma(v)}}
\\ =
XB(G;t-1, x_0,\dots)|_{x_i= r^{q^i}\text{ for each }i< k\text{ and for }i\geq k, x_i= 0}.
\end{multline*}
This shows that $B_{r,q}(G)$ is an evaluation of $XB(G)$.  

For $k\in \BN$, $b\in \BN$, and  $c= (c_i\geq 0, i=0, \ldots, k-1)$ such that $\sum_i c_i= |V|$, the coefficient $K(k,b,c)$ of 
$t^b\prod_{0\leq i< k}x_i^{c_i}$ in $XB(G; t-1, x_0,\ldots)$ is clearly equal to the coefficient of  $t^b\prod_{0\leq i< k}x_i^{c_i}$ in $XB(G; t-1, x_0,\ldots x_{k-1}, 0,0,\ldots)$. In order to determine $XB(G)$, it suffices to determine each coefficient $K(k,b,c)$ of $t^b\prod_{0\leq i< k}x_i^{c_i}$ in $XB(G; t-1, x_0,\ldots)$. 

If the coefficient of $t^b\prod_{0\leq i< k}r^{c_iq^i}$ in $B_{r,q}(G;t-1,k)$ is uniquely determined, then clearly it can be calculated from  the evaluations of $B_{r,q}(G;t-1,k)$ and it is equal to $K(k,b,c)$. Hence it suffices to show the following:
\newline
 {\bf Claim.} For each $k, b\in \mathbb{N}$, the coefficient of $t^b$ in $B_{r,q}(G;t-1,k)$ may be uniquely written in the form 
$$
 \sum_{c= (c_i\geq 0, i=0, \ldots, k-1);\sum c_i= |V|} a_cr^{\sum_ic_iq^i},
$$
where $a_c\in \mathbb{N}$ for each $c$.
We prove the Claim  by contradiction:
If it is not true then for some non-empty finite set $\s$ of vectors $\{ (c_i\geq 0, i=0, \ldots, k-1); \sum c_i= |V|\}$,  and non-zero integers $a_c\neq 0, c\in \s$, the function 


$$
F(q, \s)= \sum_{c= (c_i\geq 0, i=0, \ldots, k-1)\in \s} a_cr^{\sum_ic_iq^i}
$$
is identically zero. This cannot happen since, if $c^*$ is the largest vector of $\s$ in the right-lexicographic ordering, then for $q$ sufficiently large, 
$$
|a_{c^*}r^{\sum_ic^*_iq^i}|> \sum_{c^*\neq c\in \s} |a_cr^{\sum_ic_iq^i}|.
$$
This completes the proof of the claim, and hence  $B_{r,q}$ determines $XB$, completing the proof of the theorem.
\end{proof}

As a corollary we obtain the equivalence of  $B_{r,q}$ and $U$.
\begin{corollary}
Let $r> 1$ be given. Then graph functions $B_{r,q}$ and $U$ are equivalent.
\end{corollary}
\begin{proof}
This follows from Theorem~\ref{thm.eqq}, and Theorem~6.2 of \cite{NW}  which gives the equivalence of $U$ and $XB$. 
\end{proof}

\bibliographystyle{amsplain} 

\end{document}